\documentclass{amsart}

\usepackage{amssymb,amsmath,amsthm,latexsym,stmaryrd}
\usepackage[pagebackref,colorlinks=true]{hyperref}
\usepackage{color}

\newtheorem{theorem}{Theorem}[section]
\newtheorem{proposition}{Proposition}[section]

\newtheorem{definition}{Definition}[section]

\newcommand{\f}{\phi}
\newcommand{\g}{\tilde{g}}
\newcommand{\n}{\nabla}
\newcommand{\M}{(\mathcal{M},\A\f,\A\xi,\A\eta,\A{}g)}

\newcommand{\R}{\mathbb R}
\newcommand{\X}{\mathfrak X}

\newcommand{\LL}{\mathcal{L}}
\newcommand{\ttt}{\tilde\tau}
\newcommand{\DD}{\partial}

\newcommand{\lm}{\lambda}
\newcommand{\al}{\alpha}

\newcommand{\A}{\allowbreak{}}
\newcommand{\D}{\mathrm{d}\hspace{-0.5pt}}
\newcommand{\thmref}[1]{Theorem~\ref{#1}}

\newcommand{\propref}[1]{Proposition~\ref{#1}}

\DeclareMathOperator{\Div}{div} % the divergence
\DeclareMathOperator{\tr}{tr} % the trace
\begin{document}

\vspace{2cm}

\title[Para-Ricci-like solitons with arbitrary potential on \dots]
{Para-Ricci-like solitons with arbitrary potential on para-Sasaki-like Riemannian $\Pi$-manifolds}

\author[H. Manev]{Hristo Manev}
\address[H. Manev]{Medical University -- Plovdiv,
Faculty of Pharmacy,
Department of Medical Physics and Biophysics,
15A Vasil Aprilov Blvd,
Plovdiv 4002, Bulgaria}
\email{hristo.manev@mu-plovdiv.bg}
\author[M. Manev]{Mancho Manev
}

%%% ENTER AUTHOR(S) AFFILIATION(S)
\address[M. Manev]{
University of Plovdiv Paisii Hilendarski,
Faculty of Mathematics and Informatics,
Department of Algebra and Geometry,
24 Tzar Asen St.,
Plovdiv 4000, Bulgaria
\&
Medical University -- Plovdiv,
Faculty of Pharmacy,
Department of Medical Physics and Biophysics,
15A Vasil Aprilov Blvd,
Plovdiv 4002, Bulgaria}
\email{mmanev@uni-plovdiv.bg}


%\date{January 1, 2007 and, in revised form, June 22, 2007.}

%\dedicatory{{\rm Dedicated to my doctoral advisor Prof. Dimitar Mekerov on the occasion of his 70th birthday}}
%\newcommand{\FilSupport}{Research supported by the project NI13-FMI-002 of the Scientific
%Research Fund at the University of Plovdiv}


\begin{abstract}
Para-Ricci-like solitons with arbitrary potential on para-Sasaki-like Riemannian $\Pi$-manifolds are introduced and studied. For the studied soliton, it is proved that its Ricci tensor is a constant multiple of the vertical component of both metrics. Thus, the corresponding scalar curvatures of both considered metrics are equal and constant. An explicit example of the Lie group as the manifold under study is presented.
\thanks{The research of H.M. is partially supported by the project MU21-FMI-008 of the Scientific Research Fund, University of Plovdiv Paisii Hilendarski. The research of M.M. is partially supported by projects MU21-FMI-008 and FP21-FMI-002 of the Scientific Research Fund, University of Plovdiv Paisii Hilendarski.}
\end{abstract}
\maketitle

\section{Introduction}\label{sec1}

The concept of Ricci solitons as a special solution of the Ricci flow equation was introduced in \cite{Ham82}. After that, a detailed study on Riemannian Ricci solitons was performed in \cite{Cao}.
The beginning of the study of Ricci solitons in contact Riemannian geometry was presented in \cite{Shar}, followed by investigations into the Ricci solitons in different types of almost contact metric manifolds \cite{GalCra,IngBag,NagPre}.
In next years, this concept was generalized and studied: in paracontact geometry \cite{Bla15,PraHad} and in pseudo-Riemannian geometry \cite{BagIng12,BlaPer,Bro-etal,MM-Sol1,MM-Sol2,MM-Sol3}.
The Ricci solitons (called quasi-Einstein) are also objects of interest for physicists (e.g., \mbox{see \cite{ChaVal96,Fri85}).}

The object of investigation in our work is the geometry of almost paracontact almost paracomplex Riemannian manifolds, which we briefly call Riemannian $\Pi$-manifolds. The induced almost product structure on the paracontact distribution of these manifolds is traceless, and the restriction on the paracontact distribution of the almost paracontact structure is an almost paracomplex structure. These manifolds are introduced in \cite{ManSta}, where they are called almost paracontact Riemannian manifolds of type $(n,n)$. After that, their investigation continues under the name ``almost paracontact almost paracomplex Riemannian manifolds'' (e.g., see \cite{ManTav57,ManTav2}).

{
A pseudo-Riemannian manifold $(M,g)$ admits a \emph{Ricci soliton} if its Ricci tensor $\rho$ has the following form \cite{Ham82}:
\begin{equation*}\label{Rs}
\begin{array}{l}
\rho = -\frac12 \mathcal{L}_v g - \lm\, g,
\end{array}
\end{equation*}
where $\mathcal{L}$ denotes the Lie derivative, $v$ is a vector field, and $\lm$ is a constant.
In the geometry of almost contact metric manifolds, a generalization of the Ricci soliton called \emph{$\eta$-Ricci soliton} is defined in \cite{ChoKim} by the following equation:
\begin{equation*}\label{eRs}
\begin{array}{l}
\rho = -\frac12 \mathcal{L}_v g  - \lm\, g  - \nu\, \eta\otimes\eta,
\end{array}
\end{equation*}
where $\nu$ is also a constant.
Due to the presence of two associated metrics, $g$ and $\g$, on the Riemannian $\Pi$-manifolds,
the question arises about the generalization of these notions in the considered geometry. We started the investigation of the so-called para-Ricci-like solitons in \cite{HM17,HM18}, where their potentials were the Reeb vector field and a vector field, which were pointwise collinear to the Reeb vector field, respectively.
In the present work, we continue with further generalizations of these notions by taking their potential to \mbox{be arbitrary.}
}

The paper is structured as follows.
In Section \ref{sec1}, we present introductory information about the studied problem to the reader.
In Section \ref{sec2}, we recall preliminary facts concerning para-Sasaki-like Riemannian $\Pi$-manifolds.
Section \ref{sec3} is devoted to the investigations of para-Ricci-like solitons with an arbitrary potential. We prove assertions for the soliton's constants and a property of its potential, as well as the fact that the Ricci tensor of the considered manifolds is a constant multiple of the vertical component of both metrics. Moreover, the corresponding scalar curvatures of both considered metrics are equal and constant.
We find the values for the sectional curvatures of the special 2-planes with respect to the considered structure for the case of dimension 3. Finally, in Section \ref{sec4}, we give an example to illustrate the obtained results.

\section{Para-Sasaki-like Riemannian $\Pi$-Manifolds}\label{sec2}

Let $\M$ be a Riemannian $\Pi$-manifold, i.e., $\mathcal{M}$ is a $(2n+1)$-dimensional differentiable manifold, $g$ is a Rie\-mannian metric, $\f$ is a (1,1)-tensor field, $\xi$ is a Reeb vector field, and $\eta$ is its dual 1-form. The structure $(\f,\xi,\eta)$ is an almost paracontact structure, and the following conditions are valid:
\begin{equation}\label{strM}
\begin{array}{c}
\f\xi = 0,\qquad \f^2 = I - \eta \otimes \xi,\qquad
\eta\circ\f=0,\qquad \eta(\xi)=1,\\ \
\tr \f=0,\qquad g(\f x, \f y) = g(x,y) - \eta(x)\eta(y),
\end{array}
\end{equation}
where $I$ is the identity transformation on $T\mathcal{M}$ \cite{ManTav57,Sato76}. The following properties are immediately derived from these basic identities:
\begin{equation}\label{strM2}
\begin{array}{ll}
g(\f x, y) = g(x,\f y),\qquad &g(x, \xi) = \eta(x),
\\
g(\xi, \xi) = 1,\qquad &\eta(\n_x \xi) = 0,
\end{array}
\end{equation}
where {$\n$ denotes} the Levi--Civita connection of $g$.
Here and further, $x$, $y$, $z$, $w$ stand for arbitrary vector fields from $\X(\mathcal{M})$ or vectors in $T\mathcal{M}$ at a fixed point of $\mathcal{M}$.

The associated metric $\g$ of $g$ on $\M$, defined by $\g(x,y)=g(x,\f y)+\eta(x)\eta(y)$, is an indefinite metric of the signature $(n + 1, n)$, and it is compatible with $\M$ {in the same way} as $g$.

In \cite{IvMaMa2}, the class of para-Sasaki-like Riemannian $\Pi$-manifolds is introduced and studied. This special subclass of the considered manifolds is determined by:
\begin{equation*}\label{defSl}
\begin{array}{l}
\left(\nabla_x\f\right)y=-g(x,y)\xi-\eta(y)x+2\eta(x)\eta(y)\xi\\
\phantom{\left(\nabla_x\f\right)y}=-g(\f x,\f y)\xi-\eta(y)\f^2 x.
\end{array}
\end{equation*}
Moreover, in the same cited work, the following identities are proved to be valid for any para-Sasaki-like $\M$:
\begin{equation}\label{curSl}
\begin{array}{ll}
\n_x \xi=\f x, \qquad &\left(\n_x \eta \right)(y)=g(x,\f y),\\
R(x,y)\xi=-\eta(y)x+\eta(x)y, \qquad &R(\xi,y)\xi=\f^2y, \\
\rho(x,\xi)=-2n\, \eta(x),\qquad 				&\rho(\xi,\xi)=-2n,
\end{array}
\end{equation}
where {$R$ and $\rho$ stand} for the curvature tensor and the Ricci tensor, respectively.

On an arbitrary Riemannian $\Pi$-Manifolds $\M$, there exists a symmetric $(0,2)$-tensor $\rho^*(y,z)=g^{ij}R(e_i,y,z,\f e_j)$ associated with $\rho$ regarding $\f$.
In the case of para-Sasaki-like $\M$, the following property is valid \cite{IvMaMa2}:
\begin{equation*}\label{curf-Sl}
\begin{array}{l}
R(x,y,\f z,w)-R(x,y,z,\f w)\\[4pt]
=-\left\{g(y,z)-2\eta(y)\eta(z)\right\}g(x,\f w)
-\left\{g(y,w)-2\eta(y)\eta(w)\right\}g(x,\f z)\\[4pt]
+\left\{g(x,z)-2\eta(x)\eta(z)\right\}g(y,\f
w)+\left\{g(x,w)-2\eta(x)\eta(w)\right\}g(y,\f z).
\end{array}
\end{equation*}
Then, taking the trace of the latter equality for $x=e_i$ and $w=e_j$, we ascertain that the tensors $\rho^*$ and $\rho$ are related as follows:
\begin{equation}\label{rho*-Sl}
\rho^*(y,z)=\rho(y,\f z)+(2n-1)g(y,\f z).
\end{equation}

\begin{proposition}\label{prop:nQ-Sl}
On a $(2n+1)$-dimensional para-Sasaki-like manifold $\M$, the following properties of the Ricci operator $Q$ are valid:
\begin{gather}\label{nQxiQ-aSl}
(\n_x Q)\xi=-Q\f x+2n\,\f x,
\\[4pt]
\label{nxiQ=Q-aSl}
(\n_\xi Q)y=-2Q\f y,
\\[4pt]
\label{etaQxi-aSl}
\begin{array}{l}
\eta\bigl((\n_x Q)\xi\bigr)=0,\qquad
\eta\bigl((\n_\xi Q)y\bigr)=0.
\end{array}
\end{gather}
\end{proposition}
\begin{proof}
Taking into account \eqref{curSl}, then $Q\xi=-2n\,\xi$ and $\n_x \xi=\f x$ hold. By virtue of these, we immediately obtain the covariant derivative in \eqref{nQxiQ-aSl}.

After that, we apply $\n_z$ to the expression of $R(x,y)\xi$ in \eqref{curSl},
and using the form of $\n\eta$ in \eqref{curSl}, we obtain the following:
\[
\left(\n_z R\right)(x,y)\xi=-R(x,y)\f z + g(y,\f z)x - g(x,\f z)y.
\]
Taking the trace of the latter equality for $z=e_i$ and $x=e_j$ and using \eqref{rho*-Sl}, we obtain the equation below:
\begin{equation*}\label{tr12}
g^{ij} (\n_{e_i} R)(e_j,y)\xi=Q\f y + 2n \f y.
\end{equation*}
By virtue of the following consequence of the second Bianchi identity:
\[
g^{ij} (\n_{e_i} R)(\xi,y)e_j=\left(\n_y Q\right)\xi-\left(\n_\xi Q\right)y
\]
and using the symmetries of $R$ and \eqref{nQxiQ-aSl}, we prove the truthfulness of \eqref{nxiQ=Q-aSl}.

As consequences of \eqref{nQxiQ-aSl} and \eqref{nxiQ=Q-aSl}, we get \eqref{etaQxi-aSl}.
\end{proof}

Let us recall from \cite{HM17} that a $\M$ is said to be
para-Ein\-stein-like with constants $(a,b,c)$ if $\rho$ satisfies the following equation:
\begin{equation*}\label{defEl}
\begin{array}{l}
\rho=a\,g +b\,\g +c\,\eta\otimes\eta.
\end{array}
\end{equation*}
In the case when $b=0$ or $b=c=0$, the manifold is called an $\eta$-Einstein manifold or an Einstein manifold, respectively. If $a$, $b$, $c$ are functions on $\mathcal{M}$, $\M$ is called almost para-Einstein-like, almost $\eta$-Einstein manifold, or an almost Einstein manifold in the respective cases.

%Let is consider a $(2n+1)$-dimensional Riemannian $\Pi$-manifold $\M$ which is para-Sasaki-like and para-Einstein-like with constants $(a,b,c)$. Tracing \eqref{defEl} and using the last equalities of \eqref{curSl}, we have: \cite{HM17}
%\begin{equation}\label{tauElSl2}
%a+b+c=-2n,\qquad \tau=2n(a-1),
%\end{equation}
%where $\tau$ stands for the scalar curvature with respect to $g$ of $\M$.
%Moreover, for the scalar curvature $\tilde\tau$ with respect to $\g$ on $\M$ we obtain
%\begin{equation}\label{abctau*-ElSl}
%\tilde\tau=2n(b-1).
%\end{equation}
%
%Taking into account \eqref{tauElSl2} and \eqref{abctau*-ElSl}, expression \eqref{defEl} gets the following form
%\begin{equation*}\label{defElSl}
%\begin{array}{l}
%\rho=\left(\dfrac{\tau}{2n}+1\right)g +\left(\dfrac{\tilde\tau}{2n}+1\right)\g
%+\left(-2(n+1)-\dfrac{\tau+\tilde\tau}{2n}\right)\eta\otimes \eta.
%\end{array}
%\end{equation*}

Let us recall the following:
\begin{proposition}[\cite{HM18}]\label{prop:El-Dtau}
Let $\M$ be a $(2n+1)$-dimensional para-Sasaki-like Riemannian $\Pi$-manifold. If $\M$ is almost para-Einstein-like  with functions $(a,b,c)$,
then the scalar curvatures $\tau$ and $\tilde\tau$ are constants,
\begin{equation*}\label{El-Dtauxi}
\tau = const, \qquad \tilde\tau=-2n
\end{equation*}
and $\M$ is $\eta$-Einstein with the following constants:
\[
(a,b,c)=\left(\frac{\tau}{2n}+1,\,0,\,-2n-1-\frac{\tau}{2n}\right).
\]
\end{proposition}

Let us remark that $\tau$ and $\tilde\tau$ stand for the scalar curvatures on $\M$ regarding $g$ and $\g$, respectively.

\section{Para-Ricci-Like Solitons with Arbitrary Potential on Para-Sasaki-Like Manifolds}\label{sec3}

In \cite{HM17}, the notion of the para-Ricci-like soliton with the potential $\xi$ and the constants $(\lm,\mu,\nu)$ is introduced:
{
\begin{definition}[\cite{HM17}]
{A Riemannian $\Pi$-manifold } $\M$ admits a
para-Ricci-like soliton with a potential vector field $\xi$ and constants $(\lm,\mu,\nu)$ if its Ricci tensor $\rho$ satisfies the following:
\begin{equation}\label{defRl}
\begin{array}{l}
\rho=-\frac12 \mathcal{L}_{\xi} g - \lm\, g - \mu\, \g - \nu\, \eta\otimes \eta.
\end{array}
\end{equation}
\end{definition}
}

{
In \cite{HM18}, similarly to the latter definition, the more general case is introduced. There,  the potential of the soliton is pointwise collinear with $\xi$.
\begin{definition}[\cite{HM18}]
{A Riemannian $\Pi$-manifold } $\M$ admits a para-Ricci-like soliton with potential vector field $\zeta=k\,\xi$ and constants $(\lm,\mu,\nu)$, where $k$ is a differentiable function on $\mathcal{M}$ if its Ricci tensor $\rho$ satisfies the following:
\begin{equation}\label{defRl-v1}
\begin{array}{l}
\rho=-\frac12 \mathcal{L}_{\zeta} g - \lm\, g - \mu\, \g - \nu\, \eta\otimes \eta.
\end{array}
\end{equation}
\end{definition}
}

Now, we generalize {the latter notion} even more and consider the case when the potential is an arbitrary vector field, as follows.

\begin{definition}
{A Riemannian $\Pi$-manifold } $\M$ admits a para-Ricci-like soliton with a potential vector field $v$ and constants $(\lm,\mu,\nu)$ if its Ricci tensor $\rho$ satisfies the following:
\begin{equation}\label{defRl-v}
\begin{array}{l}
\rho=-\frac12 \mathcal{L}_{v} g - \lm\, g - \mu\, \g - \nu\, \eta\otimes \eta.
\end{array}
\end{equation}
\end{definition}

If $\mu=0$ or $\mu=\nu=0$, then \eqref{defRl-v} {(similarly, \eqref{defRl} and \eqref{defRl-v1} for the other potentials)} defines an $\eta$-Ricci soliton or a Ricci soliton on $\M$, respectively. For functions $\lm$, $\mu$, $\nu$ on $\mathcal{M}$, the soliton is called almost para-Ricci-like soliton, almost $\eta$-Ricci soliton, or almost Ricci soliton in the respective cases.

\begin{theorem}\label{thm:RlSl-v}
Let $\M$ be a $(2n+1)$-dimensional para-Sasaki-like Riemannian $\Pi$-manifold, and let it admit a para-Ricci-like soliton with an arbitrary potential vector field $v$ and constants $(\lm,\mu,\nu)$. Then, the following identities are valid:
\begin{equation}\label{lmn-v}
\lm+\mu+\nu=2n,
\end{equation}
\begin{equation*}\label{nxiv}
\n_\xi v=\f v.
\end{equation*}
\end{theorem}

\begin{proof}
Taking into account \eqref{curSl} and \eqref{defRl-v}, the covariant derivative of $\bigl(\LL_v g\bigr)(y,z)$ with respect to $\n_x$  has the following form:
\begin{equation}\label{nLvg}
\begin{array}{l}
\bigl(\n_x \LL_v g\bigr)(y,z) = -\,2 \bigl(\n_x \rho\bigr)(y,z)
-2\mu\{g(\f x,\f y)\eta(z)+g(\f x,\f z)\eta(y)\}\\[4pt]
\phantom{\bigl(\n_x \LL_v g\bigr)(y,z) =}
-2(\mu+\nu)\{g( x,\f y)\eta(z)+g( x,\f z)\eta(y)\}.
\end{array}
\end{equation}

In \cite{Yano70}, it is proved that for a metric connection $\n$,
\[
\bigl(\n_x \LL_v g\bigr)(y,z) =
g\left((\LL_v \n)(x,y),z\right) + g\left((\LL_v \n)(x,z),y \right).
\]
Using the symmetry of $\LL_v \n$, the latter formula can be read as follows:
\begin{equation}\label{gLvn}
2g\bigl((\LL_v \n)(x,y),z\bigr) = \bigl(\n_x \LL_v g\bigr)(y,z) +
\bigl(\n_y \LL_v g\bigr)(z,x) - \bigl(\n_z \LL_v g\bigr)(x,y).
\end{equation}

Applying \eqref{gLvn} to \eqref{nLvg}, we get the following equation:
\begin{equation}\label{nLvg2}
\begin{array}{l}
g\bigl((\LL_v \n)(x,y),z\bigr)  =
- \bigl(\n_x \rho\bigr)(y,z)- \bigl(\n_y \rho\bigr)(z,x) + \bigl(\n_z \rho\bigr)(x,y) \\[4pt]
\phantom{g\bigl((\LL_v \n)(x,y),z\bigr) =}
-2\mu\,g(\f x,\f y)\eta(z)-2(\mu+\nu)g( x,\f y)\eta(z).
\end{array}
\end{equation}
Now, setting $y=\xi$ in \eqref{nLvg2} and using \eqref{nQxiQ-aSl} and \eqref{nxiQ=Q-aSl}, we obtain the following:
\begin{equation}\label{Lvnxi}
(\LL_v \n)(x,\xi) = 2 Q\f x.
\end{equation}
Then, the covariant derivative of \eqref{Lvnxi}, by means of \eqref{curSl}, has the following form:
\begin{equation}\label{nLvn}
\begin{array}{l}
\bigl(\n_y \LL_v \n\bigr)(x,\xi) = -(\LL_v \n)(x,\f y) + 2(\n_y Q)\f x - 2\eta(x) Qy\\[4pt]
\phantom{\bigl(\n_y \LL_v \n\bigr)(x,\xi) =}
+4n\,g(x,y)+2(2n+1)\eta(x)\eta(y)\xi.
\end{array}
\end{equation}

Now, we apply \eqref{nLvn} to the following formula from \cite{Yano70}:
\begin{equation}\label{LvR}
(\LL_v R)(x,y)z = \bigl(\n_x \LL_v \n\bigr)(y,z) - \bigl(\n_y \LL_v \n\bigr)(x,z)
\end{equation}
and bearing in mind the symmetry of $\LL_v \n$, we get
the following consequence of \eqref{Lvnxi}--\eqref{LvR}:
\begin{equation}\label{LvRxyxi}
\begin{array}{l}
g\bigl((\LL_v R)(x,y)\xi,z\bigr) =  \,(\n_x \rho)(\f y,z)-(\n_{\f y} \rho)(x,z) +(\n_z \rho)(x,\f y)\\[4pt]
\phantom{g\bigl((\LL_v R)(x,y)\xi,z\bigr) =}
  -(\n_y \rho)(\f x,z)+(\n_{\f x} \rho)(y,z) -(\n_z \rho)(\f x, y)\\[4pt]
\phantom{g\bigl((\LL_v R)(x,y)\xi,z\bigr) =}
+2\,\eta(x)\rho(y,z)-2\,\eta(y)\rho(x,z).
\end{array}
\end{equation}
Substituting $y=z=\xi$ in \eqref{LvRxyxi} and using \eqref{Lvnxi}, we obtain the equation below:
\begin{equation}\label{LvRQ}
(\LL_v R)(x,\xi)\xi = 0.
\end{equation}

On the other hand, by applying $\LL_v$ to the expression of $R(x,\xi)\xi$ from \eqref{curSl}
and using \eqref{defRl-v}, we get the following equation:
%\begin{equation*}\label{LvR-Sl}
%(\LL_v R)(x,\xi)\xi = (\LL_v \eta)(x)\xi+g(x,\LL_v \xi)\xi-2 \eta(\LL_v \xi)x
%\end{equation*}
%or in an equivalent form
\begin{equation}\label{LvR2-Sl}
(\LL_v R)(x,\xi)\xi = \{(\LL_v \eta)(x)+g(x,\LL_v \xi)-2 \eta(\LL_v \xi)\eta(x)\}\xi+2 \eta(\LL_v \xi)\f^2x.
\end{equation}
When we compare \eqref{LvRQ} and \eqref{LvR2-Sl}, we obtain the following system of equations:
%\[
%(\LL_v \eta)(x)+g(x,\LL_v \xi)-2 \eta(\LL_v \xi)\eta(x)=0,\qquad \eta(\LL_v \xi)=0,
%\]
%i.e.
\begin{equation}\label{LvR3-Sl}
(\LL_v \eta)(x)+g(x,\LL_v \xi)=0,\qquad \eta(\LL_v \xi)=0.
\end{equation}

For a para-Sasaki-like manifold $\M$, using \eqref{defRl-v} and $\rho(x,\xi)=-2n\eta(x)$ \mbox{from \eqref{curSl},} we have the equation below:
\begin{equation}\label{Lvgxxi-SlRl}
(\LL_v g)(x,\xi)=-2 (\lm+\mu+\nu-2n)\eta(x).
\end{equation}
By plugging $x=\xi$ into the latter equality, we get the following:
\begin{equation}\label{Lvgxixi1-SlRl}
(\LL_v g)(\xi,\xi)=-2 (\lm+\mu+\nu-2n).
\end{equation}

From another point of view, for the Lie derivative of $g(x,\xi)=\eta(x)$ with respect to $v$, we have the following:
\begin{equation}\label{Lvgxxi}
(\LL_v g)(x,\xi)=(\LL_v \eta)(x)-g\left(x,\LL_v \xi\right),
\end{equation}
which for $x=\xi$ induces the following equation:
\begin{equation}\label{Lvgxixi2-SlRl}
(\LL_v g)(\xi,\xi)=-2\eta\left(\LL_v \xi\right).
\end{equation}
Therefore, comparing \eqref{Lvgxixi1-SlRl} and \eqref{Lvgxixi2-SlRl}, we obtain:
\[
\eta(\LL_v \xi)=\lm+\mu+\nu-2n,
\]
which implies \eqref{lmn-v} by virtue of \eqref{LvR3-Sl}.

Substituting \eqref{lmn-v} in \eqref{Lvgxxi-SlRl} gives that $(\LL_v g)(x,\xi)=0$. Therefore, we have\linebreak \mbox{$(\LL_v \eta)(x)=g\left(x,\LL_v \xi\right)$}, while taking into account \eqref{Lvgxxi}.
Hence, bearing in mind \eqref{LvR3-Sl}, we get the vanishing of $\LL_v \xi$,
which together with $\n\xi=\f$ from \eqref{curSl} completes the proof.
\end{proof}

\begin{proposition}\label{prop:Lvrho}
Let $\M$ be a $(2n+1)$-dimensional para-Sasaki-like Riemannian $\Pi$-manifold. If $\M$ admits a para-Ricci-like soliton with an arbitrary potential $v$, then $\rho$, $\tau$, and $\tilde\tau$ satisfy the following identities:
\begin{equation*}\label{Lvrhoxi=}
\left(\LL_v \rho\right)(x,\xi) = 0, \qquad \tau=-2n, \qquad \ttt=const.
\end{equation*}
\end{proposition}
\begin{proof}
By means of \eqref{LvR}, we get the following equations:
\begin{equation*}\label{LvRxi}
\begin{array}{l}
g\bigl(\left(\LL_v R\right)(x,y)\xi,z\bigr) = \, g\bigl(\left(\LL_v \n\right)(x,\f y),z\bigr)
- g\bigl(\left(\LL_v \n\right)(\f x, y),z\bigr)
\\[4pt]
\phantom{g\bigl(\left(\LL_v R\right)(x,y)\xi,z\bigr) =}
+ 2 \left(\n_x \rho\right)(\f y, z) - 2 \left(\n_y \rho\right)(\f x, z)
\\[4pt]
\phantom{g\bigl(\left(\LL_v R\right)(x,y)\xi,z\bigr) =}
+ 2 \eta(x)\rho(y,z) - 2 \eta(y)\rho(x,z).
\end{array}
\end{equation*}
Taking the trace of the latter equality for $x=e_i$ and $z=e_j$ and bearing in mind \eqref{nLvg2} and
$\D\tau=2\Div\rho$, we obtain the following:
\begin{equation*}
\begin{array}{l}
g^{ij}g\bigl(\left(\LL_v \n\right)(e_i,\f y),e_j\bigr) = - \D\tau(\f y),
\\[4pt]
g^{ij}g\bigl(\left(\LL_v \n\right)(\f e_i, y),e_j\bigr) = \D\ttt(y),\\[-6pt]
\end{array}
\end{equation*}
\begin{equation}\label{trLvR}
g^{ij}g\bigl(\left(\LL_v R\right)(e_i,y)\xi,e_j\bigr) =  \left(\LL_v \rho\right)(y,\xi).
\end{equation}
Therefore, the following equality is valid:
\begin{equation}\label{LvrhoxiD}
\left(\LL_v \rho\right)(y,\xi) = \, \D\ttt(y)-2(\tau+2n)\eta(y).
\end{equation}
By substituting $y=\xi$, we get:
\begin{equation}\label{Lvrhoxixi}
\left(\LL_v \rho\right)(\xi,\xi) = \, \D\ttt(\xi)-2(\tau+2n).
\end{equation}

From another point of view, due to \eqref{LvRQ} and \eqref{trLvR}, we have $\left(\LL_v \rho\right)(\xi,\xi)=0$. Therefore, \eqref{LvrhoxiD} and \eqref{Lvrhoxixi} deduce the equations below:
\begin{equation}\label{Lvrhoxi}
\left(\LL_v \rho\right)(x,\xi) = \D\tilde\tau(\f^2x), \qquad \D\tilde\tau(\xi) = 2(\tau + 2n).
\end{equation}
The latter equalities together with \eqref{Lvrhoxi} imply consecutively $\D\ttt(\xi)=0$ and
\[
\tau=2n, \qquad \ttt=const,
\]
which prove the assertion.
\end{proof}

\begin{theorem}\label{thm:ElSlRl}
Let $\M$ be a $(2n+1)$-dimensional para-Einstein-like para-Sasaki-like Riemannian $\Pi$-manifold.
If it admits a para-Ricci-like soliton with an arbitrary potential $v$, then the following identities are valid:
\begin{equation}\label{rho-tau-ttau}
\rho=-2n\, \eta\otimes\eta, \qquad \tau=\ttt=-2n.
\end{equation}
\end{theorem}
\begin{proof}
The assertions in \eqref{rho-tau-ttau} follow from \thmref{thm:RlSl-v}, \propref{prop:El-Dtau}, and \propref{prop:Lvrho}.
\end{proof}

%\begin{corollary}\label{cor:012n}
%Let $\M$, $\dim M = 2n+1$, be an Einstein-like Sasaki-like manifold.
%Then it is $\eta$-Einstein with constants $(0,0,2n)$,
%which is equivalent to the existence on $M$ of a Ricci-like soliton
%with potential $\xi$ and constants $(0,1,-2n-1)$.
%\end{corollary}
%\begin{proof}
%Using \thmref{thm:ElSlRl}, we obtain the following expression
%$
%\LL_v g =-2\lm\,g-2\mu\,\g$ $+2(\lm+\mu)\eta\otimes\eta,
%$
%which holds for $\lm=0$, $\mu=1$ in the case $v=\xi$.
%Therefore \thmref{thm:RlSl} is restricted to its case (iii) and $a=0$.
%\end{proof}

Let us recall that the sectional curvature $k(\alpha;p)$ of an arbitrary non-degenerate 2-plane $\alpha$, with a basis $\{x,y\}$ regarding $g$ in
$T_p\mathcal{M}$, $p \in \mathcal{M}$, is determined by the following equation:
\begin{equation}\label{sect}
k(\alpha;p)=\frac{R(x,y,y,x)}{g(x,x)g(y,y)-[g(x,y)]^2}.
\end{equation}
For $\dim \mathcal{M}=3$, we distinguish two special types of $2$-planes $\alpha$ with respect to the structure: a $\f$-holomorphic section ($\alpha= \f \alpha$) and a $\xi$-section ($\xi \in \alpha$).
Let us note that every $\f$-holomorphic section has a basis with the form $\{\f x,\f^2 x\}$.

\begin{theorem}\label{thm:dim3}
Let $\M$ be a $3$-dimensional para-Sasaki-like Riemannian $\Pi$-manifold.
If it admits a para-Ricci-like soliton with an arbitrary potential $v$, then:
\begin{itemize}
	\item[(i)] The sectional curvatures of its $\f$-holomorphic sections are equal to $1$;
   \item[(ii)] The sectional curvatures of its $\xi$-sections are equal to $-1$.
\end{itemize}
\end{theorem}
\begin{proof}
A well-known fact is that the curvature tensor of a 3-dimensional manifold has the following form:
\begin{equation}\label{Rxyz-3dim}
\begin{array}{l}
R(x,y,z,w) = g(y,z)\rho(x,w) - g(x,z)\rho(y,w) \\[4pt]
\phantom{R(x,y,z,w) = }+ \rho(y,z)g(x,w) - \rho(x,z)g(y,w)\\[4pt]
\phantom{R(x,y,z,w) = }-\dfrac{\tau}{2}\{g(y,z)g(x,w) - g(x,z)g(y,w)\}.
\end{array}
\end{equation}
By plugging $y=z=\xi$ and taking into account \eqref{curSl}, we have the equation below:
\begin{equation*}\label{rho3dim}
\rho = \frac12\{(\tau+2)g-(\tau+6)\eta\otimes\eta\},
\end{equation*}
which means that $\M$ is an $\eta$-Einstein manifold.
Therefore, due to \thmref{thm:ElSlRl}, we have the following:
\[
\rho=-2\, \eta\otimes\eta,\qquad \tau=\ttt=-2.
\]

Substituting the latter two equalities in \eqref{Rxyz-3dim}, we get the following equations:
\begin{equation}\label{Rxyz-3dim-Rl}
\begin{array}{l}
R(x,y,z,w) = \,[g(y,z)-2\,\eta(y)\eta(z)]g(x,w) - 2\,g(y,z)\eta(x)\eta(w) \\[4pt]
\phantom{R(x,y,z,w) = \,}
-[g(x,z) -2\,\eta(x)\eta(z)]g(y,w) + 2\,g(x,z)\eta(y)\eta(w).
\end{array}
\end{equation}

Then, we calculate the sectional curvature of an arbitrary $\f$-holomorphic section with respect to a basis of $\{\f x,\f^2 x\}$, using \eqref{strM}, \eqref{strM2}, \eqref{sect} and \eqref{Rxyz-3dim-Rl}, and obtain:
\[
k(\f x,\f^2 x)=1.
\]

In a similar way, in the case of a $\xi$-section with respect to a basis $\{x,\xi\}$, we get:
\[
k(x,\xi)=-1.
\]
%
%Let $\{e_1,e_2,e_3\}$ be an orthonormal $\f$-basis with respect to $g$, i.e. the following conditions are valid
%\begin{equation}\label{str-3dim}
%\begin{array}{c}
%\f e_1=e_2,\quad \f e_2=-e_1,\quad e_3=\xi,\\[4pt]
 %g_{11}=-g_{22}=g_{33}=1,\quad g_{12}=-g_{13}=-g_{23}=0.
%\end{array}
%\end{equation}
%In $\T_p M$ at point $p\in M$, there are three basic 2-planes $\bt_{ij}$ determed by their bases $\{e_i,e_j\}$ ($i,j\in\{1,2,3\}, i\neq j$). Their sectional curvatures $k_{ij}$ are computed by the
%formula
%$
%k_{ij}=R_{ijji}/g_{ii}g_{jj}.
%$
%Then, using \eqref{Rxyz-3dim-Rl} and \eqref{str-3dim}, we obtain
%\eqref{kij-3dim-Rl}.
\end{proof}

\section{Example}\label{sec4}
Let $\mathcal{M}$ be a set of points in $\R^3$ with coordinates $(x^1,x^2,x^3)$. We equip $\mathcal{M}$ with a Riemannian $\Pi$-structure $(\f, \xi, \eta, g)$ as follows:
\begin{equation*}\label{strEx1-loc}
\begin{array}{c}
g\left(\DD_1,\DD_1\right)=g\left(\DD_2,\DD_2\right)=\cosh 2x^3,\qquad
g\left(\DD_1,\DD_2\right)=\sinh 2x^3,
\\[4pt]
g(\DD_1,\DD_3)=g(\DD_2,\DD_3)=0,\qquad g(\DD_3,\DD_3)=1,
\\[4pt]
\f  \DD_1=\DD_2,\qquad \f \DD_2=\DD_1,\qquad \xi=\DD_3,
\end{array}
\end{equation*}
where $\DD_1$, $\DD_2$, $\DD_3$ stand for
$\frac{\DD}{\DD{x^1}}$,
$\frac{\DD}{\DD{x^2}}$, $\frac{\DD}{\DD{x^3}}$, respectively.
The triplet of vectors $\{e_1,e_2,e_3\}$ determined by the equations:
\begin{equation}\label{eiloci}
e_1=\cosh x^3 \DD_1-\sinh x^3 \DD_2,\qquad
e_2=-\sinh x^3 \DD_1+\cosh x^3 \DD_2,\qquad
e_3=\DD_3
\end{equation}
forms an orthonormal $\f$-basis for $T_p\mathcal{M}$, $p\in \mathcal{M}$. Therefore, we have the following equations:
\begin{equation}\label{strEx1}
\begin{array}{c}
g(e_i,e_i)=1,\quad
g(e_i,e_j)=0,\quad
i,j\in\{1,2,3\},\; i\neq j,
\\[4pt]
\f  e_1=e_2,\qquad \f e_2=e_1,\qquad \xi=e_3.
\end{array}
\end{equation}

Using \eqref{eiloci}, we obtain the commutators of $e_i$, as follows:
\begin{equation}\label{com}
[e_0,e_1]=-e_2,\qquad
[e_0,e_2]=-e_1,\qquad [e_1,e_2]=0.
\end{equation}
In \cite{ManTav2}, it is proved that the solvable Lie group corresponding to the Lie algebra defined \mbox{by \eqref{com}} and equipped with the Riemannian $\Pi$-structure $(\f, \xi, \eta, g)$ from \eqref{strEx1}
is a para-Sasaki-like Riemannian $\Pi$-manifold. Moreover, the basic components of $\n$, $R_{ijkl}=R(e_i,e_j,e_k,e_l)$, and $\rho_{ij}=\rho(e_i,e_j)$ as well as the values of $\tau$, $\ttt$, and $k_{ij}=k(e_i,e_j)$ are obtained. Their nonzero values are determined by the following equations and the symmetry properties of $R$:
\begin{equation}\label{neij}
\n_{e_1} e_2=\n_{e_2} e_1= -e_3,\qquad
\n_{e_1} e_3= e_2,\qquad
\n_{e_2} e_3= e_1;
\end{equation}
\begin{equation*}\label{Rrho-Ex2}
R_{1221}=-R_{1331}=-R_{2332}=1;
\end{equation*}
\begin{equation}\label{rhotauk}
\rho_{33}=-2,\qquad \tau=\ttt=-2,\qquad k_{12}=-k_{13}=-k_{23}=1.
\end{equation}
From the first equality of \eqref{rhotauk}, we get that the Ricci tensor has the following form:
\begin{equation}\label{rho=2etaeta}
\rho=-2\eta\otimes\eta,
\end{equation}
i.e., $\M$ is a para-Einstein-like para-Sasaki-like Riemannian $\Pi$-manifold with constants $(a,b,c)=(0,0,-2)$. These results support Theorems \ref{thm:ElSlRl} and \ref{thm:dim3}.

Now, let us consider a vector field $v$ determined by the following equations:
\begin{equation}\label{v}
\begin{array}{l}
v=v^1e_1 + v^2e_2 + v^3e_3,
\\[4pt]
v^1=-\{c_1\cosh x^3 - c_2\sinh x^3\}x^1
+\{c_2\cosh x^3 -c_1\sinh x^3\}x^2 + \sinh x^3,
\\[4pt]
v^2=\{c_2\cosh x^3 -c_1\sinh x^3\}x^1
-\{c_1\cosh x^3 - c_2\sinh x^3\}x^2 + \cosh x^3,
\\[4pt]
v^3=c_3,
\end{array}
\end{equation}
where $c_1$, $c_2$, $c_3$ are arbitrary constants.

By virtue of \eqref{eiloci}, \eqref{strEx1}, \eqref{neij}, and \eqref{v}, we obtain the following:
\begin{equation}\label{neiv}
\begin{array}{l}
\n_{e_1} v= -c_1e_1  +(c_2+c_3)e_2 - v^2 e_3,
\\[4pt]
\n_{e_2} v= (c_2+c_3)e_1  -c_1e_2 - v^1 e_3,
\\[4pt]
\n_{e_3} v= v^2e_1  +v^1 e_2.
\end{array}
\end{equation}
Using \eqref{neiv}, we calculate the basic components $\left(\LL_v g\right)_{ij}=\left(\LL_v g\right)(e_i,e_j)$ of the Lie derivative $\LL_v g$, and the nonzero ones are the following:
\[
\begin{array}{l}
\left(\LL_v g\right)_{11}=\left(\LL_v g\right)_{22}=-2c_1,\qquad
\left(\LL_v g\right)_{12}=2(c_2+c_3).
\end{array}
\]
Therefore, the tensor $\LL_v g$ has the following form:
\begin{equation}\label{LvgEx2}
\LL_v g = -2\,c_1\,g   +2(c_2+c_3)\g   +2(c_1-c_2-c_3)\eta\otimes\eta.
\end{equation}
Substituting \eqref{rho=2etaeta} and \eqref{LvgEx2} in \eqref{defRl-v}, we deduce that $\M$ admits a para-Ricci-like soliton with a potential $v$ determined by \eqref{v}
and the constants:
\[
\lm=c_1,\qquad
\mu=-(c_2+c_3),\qquad
\nu=-(c_1-c_2-c_3-2).
\]

The latter results are in accordance with Theorems \ref{thm:RlSl-v}--\ref{thm:dim3}.

%%%%%%%%%%%%%%%%%%%%%%%%%%%%%%%%%%%%%%%%%%

\end{document}